\documentclass[11pt]{article}

\usepackage{amssymb,latexsym}
\usepackage{amsmath,amscd}
\usepackage{theorem}
\usepackage[all]{xy}
\usepackage{url}
\usepackage{stmaryrd}
\usepackage{bm} 
\usepackage{graphics}
\usepackage{color}
\usepackage{soul}

\title{\bf 
Self-similarity of the classical $\bm{p}$-adic
\\
Lie groups and Lie algebras 
}

\author{
        Karina Livramento
        \\[0.1cm]
        Francesco Noseda  
}

\date{}

\newcommand{\bb}[1]{\mathbb{#1}}

\newcommand{\mr}[1]{\mathrm{#1}}

\newcommand{\rar}{\rightarrow}

\newcommand{\map}{\longrightarrow}

\newcommand{\vep}{\varepsilon}

\newcommand{\les}{\leqslant}
\newcommand{\ges}{\geqslant}

\newcommand{\ep}{\hfill $\square$} 

\newtheorem{lem} {Lemma} [section]
\newtheorem{proposition} [lem] {Proposition}

\theorembodyfont{\rm} 

\newtheorem{remark}[lem]{Remark}

\newenvironment{proof}{{\sc Proof:}}{
\hfill $\square$}
\newenvironment{proof2}{{\sc Proof:}}{}

\numberwithin{equation}{section}

\theorembodyfont{\it}

\newtheorem{theoremx}{Theorem}

\newtheorem{problemx}[theoremx]{Problem}
\newtheorem{corollaryx}[theoremx]{Corollary}

\usepackage{array}
\usepackage[colorlinks]{hyperref}

\newcommand{\Z}{\mathbb{Z}}
\newcommand{\Q}{\mathbb{Q}}

\newcommand{\F}{\mathbb{F}}

\DeclareMathOperator{\Id}{Id}
\DeclareMathOperator{\tr}{tr}
\DeclareMathOperator{\diag}{diag}
\DeclareMathOperator{\sll}{sl}
\DeclareMathOperator{\spt}{sp}
\let\so\undefined
\DeclareMathOperator{\so}{so}
\DeclareMathOperator{\gl}{gl}

\DeclareMathOperator{\GL}{GL}
\DeclareMathOperator{\SL}{SL}
\DeclareMathOperator{\Sp}{Sp}
\DeclareMathOperator{\SO}{SO}

\renewcommand{\leq}{\leqslant}
\renewcommand{\geq}{\geqslant}

\begin{document} 

\maketitle

\begin{abstract}
We exhibit infinite lists of ramification indices $\delta$ for which
the classical Lie groups over the ring of integers of $p$-adic fields
admit a faithful self-similar action on a regular rooted $\delta$-ary tree
in such a way that the action is transitive on the first level.
These results follow from the study of virtual endomorphisms of the classical 
Lie lattices over the same type of rings.
In order to compute the ramification indices for all the types of groups treated in the paper, 
we compute the indices of principal congruence subgroups of the orthogonal groups
for a class of local rings. 
\end{abstract}

{
\let\thefootnote\relax
\let\thefootnote\relax
\let\thefootnote\relax\footnotetext{
{The first author was supported by the Fundação Carlos Chagas Filho de Amparo à Pesquisa do
Estado do Rio de Janeiro (FAPERJ) - Processo nº SEI-260003/013443/2023, and by 
the Coordenação de Aperfeiçoamento de Pessoal de Nível Superior - Brasil (CAPES) - Finance Code  001.}}
}


\section{Introduction}

Groups that admit a faithful self-similar action on some regular 
rooted tree form an interesting class that contains many important examples 
such as the Grigorchuk 2-group \cite{Gri80} and the Gupta-Sidki $p$-groups \cite{GuSi83}.
More recently, there has been an intensive study on the self-similar actions of other 
families of groups, 
including $p$-adic analytic pro-$p$ groups 
\cite{NS2019, NSGGD22, NS2022JGT, NSnorm1ArxivV2, WZ2023arxv1},
which we will focus on in this paper.
For references on the study of the self-similarity of other classes of groups
the reader may look at the references given in, for instance, \cite{NSGGD22}.

We say that a group $G$ is self-similar of index ${\delta}$,
where $\delta\ges 1$ is an integer,
if $G$ admits a faithful self-similar action on
a regular rooted $\delta$-ary tree in such a way that the action is transitive on the first level.
We say that $G$ is self-similar if it is self-similar of 
some index $\delta$; see, for instance, \cite{NekSSgrp} for a general treatment of self-similar actions. 
It is known that a group is self-similar of index $\delta$
if and only if it admits a simple virtual endomorphism of index 
$\delta$.
A {virtual endomorphism} of $G$ of index $\delta$
is a group homomorphism $\varphi:D\rar G$, where $D\les G$
is a subgroup of index $\delta$; the virtual endomorphism $\varphi$ is said to be simple
if there are no nontrivial normal subgroups $N$ of $G$ that are $\varphi$-invariant,
that is, such that $N\subseteq D$ and $\varphi(N)\subseteq N$.

In this paper we deal with the classical $p$-adic analytic Lie groups 
$\SL_n(K)$, $\Sp_n(K)$, and $\SO_n(K)$, 
where $K$ is a $p$-adic field.
We introduce the precise context in order to state the main results.
Let $n\ges 2$ be an integer, $p$ be a prime, and $K$ be a finite field extension of $\bb{Q}_p$ 
of degree $d$, ramification index $e$, and inertia degree $f$.
Let $R$ be the ring of integers of $K$, and let $\pi\in R$ be a uniformizing parameter.
We denote by $q=p^f$ the cardinality of the residue field of $R$.
For $\Sigma$ one of the symbols $\SL$, $\Sp$, and $\SO$,
and for any integer $m\ges 1$, the $m$-th principal congruence subgroup $\Sigma_n^m(R)$
of $\Sigma_n(R)$ is defined to be the kernel of the map
$\Sigma_n(R)\rar \Sigma_n(R/\pi^mR)$ of reduction modulo $\pi^m$. 
Let $l$ be the rank of $\Sigma_n(K)$, that is,
the dimension over $K$ of any Cartan subalgebra of the associated $K$-Lie algebra.
We recall that the relation between $n$ and $l$ is as follows:
$$
\begin{array}{ll}
n = l + 1 & \mbox{ for }\SL_n,
\\
n = 2l  & \mbox{ for }\Sp_n\mbox{ and }n\mbox{ even},
\\
n = 2l  & \mbox{ for }\SO_n\mbox{ and }n\mbox{ even}, 
\\
n = 2l+1  & \mbox{ for }\SO_n\mbox{ and }n\mbox{ odd}. 
\end{array}
$$ 
We also observe that $l\ges 1$.

\bigskip 
The following theorem is the main result of the paper,
and it gives lists of self-similarity indices
for the classical groups.

\begin{theoremx}
\label{tssclgpK2}
In the above context,
let $k\ges 1$ and $m\ges e$ be integers.
Then the following holds.
\begin{enumerate}
\item 
Assume that 
$(l^2+2l)d\les p$.
Then 
the compact $p$-adic analytic group
$\SL_{l+1}(R)$ is self-similar of index
$$
q^{lk+(l^2+2l)m}\prod_{i=1}^{l}
\left(1- \frac{1}{q^{i+1}}\right).
$$

\item 
Assume that 
$(2l^2+l)d\les p$.
Then 
the compact $p$-adic analytic group
$\Sp_{2l}(R)$ is self-similar of index
$$
q^{2lk+(2l^2+l)m}\prod_{i=1}^l
\left(1- \frac{1}{q^{2i}}\right).
$$

\item 
Assume that 
$l\ges 2$ and 
$(2l^2-l)d\les p$.
Then 
the compact $p$-adic analytic group
$\SO_{2l}(R)$ is self-similar of index
$$
q^{(2l-2)k+(2l^2-l)m}
\left(1- \frac{1}{q^{l}}\right)
\prod_{i=1}^{l-1}
\left(1- \frac{1}{q^{2i}}\right).
$$

\item 
Assume that 
$(2l^2+l)d\les p$.
Then 
the compact $p$-adic analytic group
$\SO_{2l+1}(R)$ is self-similar of index
$$
q^{(2l-1)k+(2l^2+l)m}\prod_{i=1}^l
\left(1- \frac{1}{q^{2i}}\right).
$$
\end{enumerate}
\end{theoremx}

Theorem \ref{tssclgpK2} is a consequence of the results
given below, 
which are interesting on their own.
More precisely, the logical flow is as follows.
One key ingredient of our approach is to linearize the problem of self-similarity by considering
the Lie algebras associated with the groups involved, as explained in the final part of 
the Introduction.
By working in the context of Lie algebras, it is possible to prove Theorem \ref{tsscllieK}
that, together with \cite[Proposition A]{NS2019}, 
implies Theorem \ref{tssclgpK}.
We observe that Theorem \ref{tssclgpK} gives self-similarity indices 
for principal congruence subgroups of the classical groups.
Next, Theorem \ref{tssclgpK} and \cite[Corollary 1.4]{NS2019} imply 
Corollary \ref{cssclgpK}.
We note that, in contrast with Theorem \ref{tssclgpK2}, 
we state Theorem \ref{tssclgpK} and Corollary \ref{cssclgpK} in terms of $n$ instead of $l$,
which has the pleasant feature of unifying the formulas for $\SO$
(the same is true for Theorem \ref{tsscllieK} and Corollary \ref{csscllieK}, the latter also about Lie algebras).
In order to deduce Theorem \ref{tssclgpK2} from Corollary \ref{cssclgpK},
one has to compute the indices, in the group theoretic sense, 
of the principal congruence subgroups of the classical groups.
For $\SL$ and $\Sp$ we rely on \cite{NewIM}, while for $\SO$
we rely on Theorem \ref{tindSO}, which is proved in Section \ref{secindSO}.
More details of
the proofs of Theorem \ref{tssclgpK} and Theorem \ref{tssclgpK2} are given in Section \ref{secpfA}.
Finally, 
in Corollary \ref{cminindex} and Problem \ref{pminindex}
we briefly deal with the question of finding the minimum integers $\delta$ such that the various classical groups 
and their principal congruence subgroups are self-similar of index $\delta$.

\begin{theoremx}
\label{tssclgpK}
In the context above Theorem \ref{tssclgpK2}, 
let $k\ges 1$ and $m\ges e$ be integers.
Then the following holds.
\begin{enumerate}
\item 
Assume that $(n^2-1)d\les p$.
Then the uniform pro-$p$ group $\SL_n^m(R)$ is self-similar of index $q^{(n-1)k}$.

\item 
Assume that $n$ is even and $(n^2+n)d\leq 2p$. 
Then the uniform pro-$p$ group 
$\Sp_{n}^m(R)$ is self-similar of index $q^{nk}$,

\item 
Assume that $n\ges 3$ and $(n^2-n)d\leq 2p$.
Then the uniform pro-$p$ group 
$\SO_n^m(R)$ is self-similar of index $q^{(n-2)k}$.
\end{enumerate}
\end{theoremx}

\begin{corollaryx}
\label{cssclgpK}
In the context above Theorem \ref{tssclgpK2}
the following holds.
\begin{enumerate}
\item 
All the open subgroups of 
$\SL_n(R)$, $\Sp_n(R)$, and $\SO_n(R)$
are self-similar,
where $n$ satisfies the assumptions of the respective item of Theorem \ref{tssclgpK}.

\item 
Let $\Sigma$ be any of the symbols $\SL$, $\Sp$, and $\SO$. 
Assume that $n$, $m$, and $k$ are integers satisfying the constraints 
of Theorem \ref{tssclgpK},
including the ones of the respective item.
Let 
$\Sigma_n^m(R)\les G\les \Sigma_n(R)$.
Then $G$ is self-similar of index $[G:\Sigma_n^m(R)]q^{F(n)k}$,
where
$$
F(n) = 
\left\{
\begin{array}{ll}
n-1 &\mbox{ if }\,\Sigma = \SL, \\
n &\mbox{ if }\,\Sigma = \Sp, \\
n-2 &\mbox{ if }\,\Sigma = \SO.
\end{array}
\right. 
$$
\end{enumerate}
\end{corollaryx}

\begin{theoremx} 
\label{tindSO}
Let $R$ be a local associative commutative ring with unit.
Let $P$ be the maximal ideal of $R$,
and assume that $P$ is generated by $\pi\in R$, where $\pi$ 
is not a zero divisor in $R$.
Assume that the residue field of $R$ is finite and of characteristic different from 2,
and let $q=|R/P|$. Let $l$ and $m$ be integers, with $l\ges 0$ and $m\ges 1$. 
Then 
\begin{eqnarray*}
[\mr{O}_{2l+1}(R) : \mr{O}_{2l+1}^m(R)] 
&=& 
2\cdot [\mr{SO}_{2l+1}(R) : \mr{SO}_{2l+1}^m(R)],
\\[3mm]
[\mr{SO}_{2l+1}(R) : \mr{SO}_{2l+1}^m(R)]
&=& 
q^{(2l^2+l)m}\prod_{i=1}^{l}\left(1 - \frac{1}{q^{2i}}\right),
\end{eqnarray*}
and, for $l\ges 1$,
\begin{eqnarray*}
[\mr{O}_{2l}(R) : \mr{O}_{2l}^m(R)] 
&=& 
2\cdot [\mr{SO}_{2l}(R) : \mr{SO}_{2l}^m(R)], 
\\[3mm]
[\mr{SO}_{2l}(R) : \mr{SO}_{2l}^m(R)] 
&=& 
q^{(2l^2-l)m}\left(1 - \frac{1}{q^l}\right)\prod_{i=1}^{l-1}\left(1 - \frac{1}{q^{2i}}\right),
\end{eqnarray*}
where $\mr{O}_n^m(R)$ and $\SO_n^m(R)$ are the kernels of the maps 
$\mr{O}_n(R)\rar \mr{O}_n(R/P^m)$ and  
$\SO_n(R)\rar \SO_n(R/P^m)$ 
of reduction modulo $P^m$.
\end{theoremx}

\bigskip
If a group $G$ is self-similar then we denote by $\delta(G)$ the minimum positive integer $\delta$
such that $G$ is self-similar of index $\delta$.

\begin{corollaryx}
\label{cminindex}
In the context above Theorem \ref{tssclgpK2}, let $m\ges e$ be an integer.
Then the following holds. 
\begin{enumerate}
\item 
Assume that 
$(l^2+2l)d\les p$.
Then 
$$
\delta(\SL_{l+1}(R))\les 
q^{l+(l^2+2l)e}\prod_{i=1}^{l}
\left(1- \frac{1}{q^{i+1}}\right)
\qquad \mbox{and}\qquad 
\delta(\SL_{l+1}^m(R))\les q^l.
$$

\item 
Assume that 
$(2l^2+l)d\les p$.
Then 
$$
\delta(\Sp_{2l}(R))\les 
q^{2l+(2l^2+l)e}\prod_{i=1}^l
\left(1- \frac{1}{q^{2i}}\right)
\qquad \mbox{and}\qquad 
\delta(\Sp_{2l}^m(R))\les q^{2l}.
$$

\item 
Assume that 
$l\ges 2$ and 
$(2l^2-l)d\les p$.
Then 
$$
\delta(\SO_{2l}(R))\les 
q^{2l-2+(2l^2-l)e}
\left(1- \frac{1}{q^{l}}\right)
\prod_{i=1}^{l-1}
\left(1- \frac{1}{q^{2i}}\right)
\qquad \mbox{and}\qquad 
\delta(\SO_{2l}^m(R))\les q^{2l-2}. 
$$

\item 
Assume that 
$(2l^2+l)d\les p$.
Then 
$$
\delta(\SO_{2l+1}(R))\les 
q^{2l-1+(2l^2+l)e}\prod_{i=1}^l
\left(1- \frac{1}{q^{2i}}\right)
\qquad \mbox{and}\qquad 
\delta(\SO_{2l+1}^m(R))\les q^{2l-1}.
$$
\end{enumerate}
\end{corollaryx}

\begin{problemx}
\label{pminindex}
Compute $\delta(G)$ for the groups $G$ that appear in Corollary \ref{cminindex}.
\end{problemx}

\bigskip

We close the Introduction by discussing the role of Lie algebras. 
The Lazard's correspondence is 
an isomorphism of categories between the category
of saturable  finitely generated pro-$p$ groups and the category of saturable $\bb{Z}_p$-Lie lattices;
see \cite[IV (3.2.6)]{Laz65}.
We recall that the Lazard's correspondence 
restricts to an isomorphism between the category
of uniform pro-$p$ groups and the category of powerful $\bb{Z}_p$-Lie lattices;
see \cite{DixAnaProP}.
In \cite{NS2019} the study of self-similar actions of $p$-adic analytic pro-$p$ groups was initiated
by translating the notion of simple virtual endomorphism from groups to Lie algebras.
We recall the relevant definitions.
A $\bb{Z}_p$-Lie lattice is a finitely generated free $\bb{Z}_p$-module 
endowed with a structure of Lie algebra over $\bb{Z}_p$.
Let $L$ be a $\bb{Z}_p$-Lie lattice and $k\in\bb{N}$.
A virtual endomorphism
of $L$ of index $p^k$ is an algebra  homomorphism $\varphi:M\rar L$
where $M$ is a subalgebra of $L$ of index $p^k$,
where the index is taken with respect to the additive group structure.
An ideal $I$ of $L$ is said to be $\varphi$-invariant if
$I\subseteq M$ and $\varphi(I)\subseteq I$.
We say that a virtual endomorphism $\varphi$ is simple if there are no nonzero ideals $I$
of $L$ that are $\varphi$-invariant.
Finally, 
we say that $L$ is {self-similar of index} ${p^k}$ if there exists
a simple virtual endomorphism of $L$ of index $p^k$,
and we say that $L$ is self-similar if 
it is self-similar of index $p^k$ for some $k$.
Proposition \cite[Proposition A]{NS2019} implies that if the uniform pro-$p$ group 
$G$ and the $\bb{Z}_p$-Lie lattice $L$ correspond to each other via the Lazard's correspondence,
and $\mr{dim}(G)\les p$, then $G$ is self-similar of index $p^k$ if and only if $L$ is self-similar of index
$p^k$. 

The $\bb{Q}_p$-Lie algebras associated with the $p$-adic analytic groups 
$\SL_n(K)$, $\Sp_n(K)$, and $\SO_n(K)$
are 
$\sll_n(K)$, $\spt_n(K)$, and $\so_n(K)$;
for more details on these Lie algebras, see the respective subsections
of Section \ref{seclieclass}.
The mentioned Lie algebras are split simple $K$-Lie algebras 
(hence, $\bb{Q}_p$-Lie algebras by restriction of scalars)
and the knowledge of the structure of such objects,
for instance, the root space decomposition, plays a prominent role.
The corresponding $R$-Lie lattices 
$\sll_n(R)$, $\spt_n(R)$, and $\so_n(R)$, 
which are also $\Z_p$-Lie lattices,
have principal congruence sublattices 
$\sigma_n^m(R)= \pi^m\sigma_n(R)$, where $m\ges 0$ is an integer
and $\sigma$ is any of the symbols in the set $\{\sll, \spt, \so\}$.
We define $\theta = 1 $ if $p\ges 3$, and $\theta = 2$ if $p=2$.
For $m\ges \theta e$, $\sigma_n^m(R)$ is a powerful Lie lattice
whose associated uniform pro-$p$ group is $\Sigma_n^m(R)$,
where $\Sigma$ is the symbol that corresponds
to $\sigma$ in the obvious sense.
See Section \ref{secpfA} for more details.

\bigskip 

The following theorem has an independent interest, and it allows to prove Theorem \ref{tssclgpK}.

\begin{theoremx}\phantom{}
\label{tsscllieK}
In the above context assume that $n\ges 2$, and let $k\ges 1$ and $m\ges 0$ be integers.
Then the following holds.
\begin{enumerate}
\item 
\label{tsscllieK1}
The $\Z_p$-Lie lattice $\sll_n^m(R)$ is self-similar of index $q^{(n-1)k}$.

\item 
\label{tsscllieK2}
For $n$ even,
the $\Z_p$-Lie lattice $\spt_{n}^m(R)$ is self-similar of index $q^{nk}$.

\item 
\label{tsscllieK3}
For $n\ges 3$,
the $\Z_p$-Lie lattice $\so_{n}^m(R)$ is self-similar of index $q^{(n-2)k}$.
\end{enumerate}
\end{theoremx}
Item (\ref{tsscllieK1}) is proved in Section \ref{ssllK},
item (\ref{tsscllieK2}) is proved in Section \ref{sspK},
item (\ref{tsscllieK3}) for $n$ even is proved in Section \ref{ssoevenK},
and 
item (\ref{tsscllieK3}) for $n$ odd is proved in Section \ref{ssooddK}.
The following corollary, analogous to Corollary \ref{cssclgpK},
is a consequence of \cite[Lemma 2.1]{NS2019}.

\begin{corollaryx}
\label{csscllieK}
\phantom{}
\begin{enumerate}
\item 
All the $\Z_p$-Lie sublattices of finite index in
$\sll_n(R)$, $\spt_n(R)$, and $\so_n(R)$
are self-similar.

\item 
Let $\sigma$ be any of the symbols $\sll$, $\spt$, and $\so$. 
Assume that $n$, $m$, and $k$ are integers satisfying the constraints 
of Theorem \ref{tsscllieK}, and let 
$\sigma_n^m(R)\subseteq L\subseteq \sigma_n(R)$,
where $L$ is a $\Z_p$-Lie sublattice of $\sigma_n(R)$.
Then $L$ is self-similar of index $[L:\sigma_n^m(R)]q^{F(n)k}$,
where $F(n)$ is defined analogously to what is done in Corollary \ref{cssclgpK}.
\end{enumerate}
\end{corollaryx}

\vspace{5mm}

\noindent 
\textbf{Notation.}
Recall that $p$ denotes a prime. The set $\bb{N}=\{0,1,2,...\}$ of natural numbers is assumed to contain $0$.
Let $R$ be an associative commutative ring with unit; 
in the applications of this paper, $R$ will be a finite field extension of 
the field $\Q_p$ of $p$-adic numbers,
or its ring of integers.
When $R$ is not a field, the terminology $R$-Lie lattice means 
a finitely generated free $R$-module
endowed with the structure of $R$-Lie algebra.
We denote by $\mr{M}_{n\times m}(R)$
the set of matrices of size $n\times m$ with entries in $R$,
and set $\gl_n(R) = \mr{M}_n(R) =\mr{M}_{n\times n}(R)$; 
the latter is an $R$-Lie algebra with respect to the commutator of matrices.
We use the notation $\Id_n$ for the $n\times n$ identity matrix,
while the transpose of a square matrix $x$ will be denoted by $x^{\mr{t}}$. 

\vspace{5mm}

\noindent 
\textbf{Acknowledgments.} 
The authors thank Ilir Snopce for suggesting the study of this topic and for comments
on early versions of this manuscript that helped improve it.


\section{Self-similarity of classical Lie lattices} 
\label{seclieclass}

For generalities on Lie algebras and root systems, 
the reader may consult, for instance, \cite{JacLieA, BouLieGrAlg3}.
We first aim to prove the technical Lemma \ref{lkarinaK},
for which we let $R$ be a domain of characteristic zero.
The lemma will be applied in cases where $R$ is the ring of integers of a $p$-adic field.
Let $\Phi$ be a root system of dimension $c$.
We will deal with matrix Lie algebras $L$ over $R$
with root systems of type $A$, $B$, $C$, and $D$;
all of these Lie algebras have the structure described as follows.
To start with, $L$ is freely generated over $R$ as an $R$-module
with basis, say, 
$(h_i:\, i\in\{1,...,c\}; e_\alpha:\, \alpha\in \Phi)$.
We denote $H=L_0=\bigoplus_{i} Rh_i$, and $L_\alpha = Re_\alpha$, so that 
$$
L= H\oplus 
\left(
\bigoplus_{\alpha\in\Phi} L_\alpha 
\right).
$$
The root system $\Phi$ is identified with a subset of $H^*=\mr{Hom}_R(H,R)$, 
and we denote by $(\vep_i)$ the dual basis of $(h_i)$ in $H^*$.
For $\gamma\in \mr{span}_\bb{Z}(\Phi) - (\Phi\cup\{0\})$
we define $L_\gamma =\{0\}\subseteq L$,
and for $\gamma \in \mr{span}_\bb{Z}(\Phi)$
we denote by $\pi_\gamma:L\rar L_\gamma$ the projection.
The Lie-algebra structure satisfies, among others, the following properties.
\begin{enumerate}
\item 
The structure constants of $L$ with respect to the basis 
$(h_i,e_\alpha)$ are integers.

\item 
$[h_i, h_j]=0$ for all $i,j$.

\item 
$[L_\alpha , L_\beta]\subseteq L_{\alpha +\beta}$
for all $\alpha,\beta\in \mr{span}_\bb{Z}(\Phi)$.

\item
For $h\in H$ and $\alpha\in \Phi$, $[h,e_\alpha]=\alpha(h)e_\alpha$.

\item
$[e_\alpha,e_{-\alpha}]\neq 0$ for all $\alpha\in \Phi$.

\item 
For $\beta,\gamma\in \mr{span}_\bb{Z}(\Phi)$,
$x\in L$, and $y\in L_\beta$
we have 
$$
\pi_\gamma([x,y])= [\pi_{\gamma-\beta}(x),y].
$$
\end{enumerate}

\begin{lem}
\label{lkarinaK}
In the above context,
let $a\in R-\{0\}$, $\Psi\subseteq \Phi$,
$$
D= H\oplus 
\left(
\bigoplus_{\alpha\in\Psi} L_\alpha 
\right),
$$
and assume that 
$$
\bigcap_{\alpha\in \Phi - \Psi} \mr{ker}(\alpha)=\{0\}.
$$
Let $I\subseteq aL$ be a subset such that $I\not\subseteq \{0\}$ and, for all 
$x\in I$ and $y\in aL$, we have $[x,y]\in I$.
Then $I\not\subseteq D$.
\end{lem}

\begin{proof2}
We will exhibit an element of $I  -  D$.
Let $x \in I  -  \{0\}$ and write $x = h + \sum\limits_{\alpha \in \Phi} x_\alpha$ 
with $h = \pi_0(x)$ and $x_\alpha = \pi_\alpha(x)$. We have the following cases.
\begin{enumerate}
 	\item \label{csni2.1} 
    If $x_\alpha \neq 0$ for some $\alpha \in \Phi- \Psi$ then $x \notin D$.
    
    \item \label{csni2.2} 
    If $h \neq 0$, by hypothesis there is some $\alpha \in \Phi-\Psi$ such that $\alpha(h)\neq 0$. 
    Then, for $y \in a L_\alpha  -  \{0\}$, we have
    $\pi_\alpha([x,y]) = [h, y] = \alpha(h)y \neq 0,$
    which proves that $[x,y] \not \in D$ by item (\ref{csni2.1}).
    Observe that, by assumption on $I$, $[x,y]\in I$.
    
    \item If $x_\alpha \neq 0$ for some 
    $\alpha \in \Psi$, we take $y \in a L_{-\alpha}  -  \{0\}$, 
    and we may apply 
    item (\ref{csni2.2}) to $[x,y] \in I$, since
    $\pi_0([x,y]) = [x_\alpha, y] \neq 0$.\ep
\end{enumerate}
\end{proof2}

\begin{remark}
\label{rDinfty}
Lemma \ref{lkarinaK} will be used to prove that the appropriate virtual endomorphisms $\varphi:M\rar L$ are simple.
Given $\varphi$ and $k\in\bb{N}$, we denote by $D_k$ the domain of the power $\varphi^k$, and we denote by 
$D_\infty(\varphi)$
the intersection of such domains; see \cite[Definition 1.1]{NS2019}.
If $I$ is a $\varphi$-invariant ideal of $L$ then it is easy to see that $I\subseteq D_\infty(\varphi)$;
see \cite[Lemma 2.1]{NS2019}.
\end{remark}

\bigskip 

In what follows,
we denote by
$e_{i,j} = e_{ij}$ 
the square matrix with 1 on entry $(i,j)$ and 0 elsewhere. 
Observe that $e_{ij}$ will represent matrices of different sizes; 
on the other hand, the size of any occurrence of such a matrix
will always be clear from the context.
Also, the notation $\mr{diag}(a_1,...,a_n)$ stands for the $n\times n$ diagonal
matrix with diagonal entries $a_1,...,a_n$.

\subsection{Self-similarity of \texorpdfstring{$\sll_{n}(R)$}{sl\_n(R)}}
\label{ssllK}

Part (\ref{tsscllieK1}) of Theorem \ref{tsscllieK} is a consequence of Proposition \ref{psssl} proven below.
We recall the context from the Introduction, in which $K$ is a finite field extension of $\bb{Q}_p$,
$R$ is the ring of integers of $K$, $\pi$ is a uniformizing parameter of $R$, and $q=|R/\pi R|$. 

We fix an integer $n\geq 2$, and we let the indices $i,j$ take values $1\les i,j\les n$. 
The $K$-Lie algebra 
$\sll_n(K)=\{x\in \gl_n(K):\, \tr(x)=0\}$ 
has a basis composed by the matrices $h_i = e_{ii} - e_{nn}$ for $i<n$ and $e_{ij}$ for $i\neq j$.
This is also a basis for the $R$-Lie lattice 
$L=\sll_n(R)=\{x\in \gl_n(R):\, \tr(x)=0\}$.
Observe that $L$ admits a decomposition such as the one described at 
the beginning of Section \ref{seclieclass} for the root system $\Phi = A_{n-1}$;
moreover, recall that the dual basis of $h_1,...,h_{n-1}$ is denoted by $\vep_1,...,\vep_{n-1}$. 
We define
$\sigma = \sum_{i<n}\varepsilon_i$.
Then the root system at hand is
\begin{gather*}
    \Phi = \{\varepsilon_i - \varepsilon_j :\,  i,j < n\mbox{ and } i\neq j \} 
    \cup \{\pm(\sigma + \varepsilon_i) :\,  i < n \}, 
\end{gather*}
while the decomposition of $L$ reads:
\begin{eqnarray*}
	H &=& \bigoplus_{i<n}R h_i,\\
    L_{\varepsilon_i - \varepsilon_j} &=& R e_{ij} \quad \text{for}\ i \neq j, \\
    L_{\sigma + \varepsilon_i} &=& R e_{in},\\
    L_{- \sigma - \varepsilon_i} &=& R e_{ni},
\end{eqnarray*}
where $i,j<n$.
Given a matrix $a = \diag(a_{1}, \dots, a_{n})$, $a_{i} \in K^{*}$, 
the $K$-Lie-algebra automorphism $\varphi(x) =axa^{-1}$ of $\sll_n(K)$ 
acts on the given basis by the rules:
$$
\begin{array}{rcll}
\varphi: \sll_{n}(K) & \to &\sll_{n}(K) & \\
h_{i} & \mapsto &h_{i} & \qquad i < n,\\
e_{ij} & \mapsto& a_{i} a_{j}^{-1} e_{ij} &\qquad  i \neq j.
\end{array}
$$
Define $M = L \cap \varphi^{-1}(L)$,
and let $m\in\bb{N}$.
Then $\pi^m M$ is an $R$-subalgebra of $\pi^mL$ of finite index, and the restriction 
$$\varphi_m:\pi^m M \to \pi^mL$$ 
of $\varphi$
is a virtual endomorphism of $\pi^mL = sl_n^m(R)$.

\begin{proposition}
\label{psssl}
In the above context, let $k\geq 1$ and take $a=\mr{diag}(1,...,1,\pi^k)$.
Then $\varphi_m$ is a simple virtual endomorphism of index $q^{(n-1)k}$ 
of the $\Z_p$-Lie lattice $\sll_n^m(R)$.
\end{proposition}

\begin{proof}
For $i < n$, we have
$\varphi(e_{ni}) = \pi^ke_{ni}$ and $\varphi(e_{in}) = \pi^{-k}e_{in}$,
and $\varphi$ acts as the identity on the remaining basis elements of ${L}$. 
We define $\Psi = \{\alpha \in \Phi :\, \sum_{i=1}^{n-1}\alpha(h_i) \leq 0\}$,
and we observe that
$\Phi-\Psi = \{\sigma+\vep_i:\,i<n\}$.
Thus 
$$
M = 
H 
\oplus 
\left[\bigoplus_{\alpha\in \Psi} Re_{\alpha }\right]
\oplus 
\left[\bigoplus_{\alpha\in \Phi - \Psi} R\pi^ke_{\alpha}\right],
$$ 
so that the quotient $L/M$ is isomorphic to 
$\bigoplus_{i<n}R/\pi^k R$
as an $R$-module.
Hence,
the index of $M$ in $L$ is $q^{(n-1)k}$ (cf. Remark \ref{rprinc}),
and we note that the index of $\pi^m M$ in $\pi^m L$ has the same value. 
We have 
$$
D_\infty(\varphi_m)\subseteq D_\infty(\varphi_0)=
H 
\oplus 
\left[\bigoplus_{\alpha\in \Psi} Re_{\alpha }\right],
$$
see Remark \ref{rDinfty} for the notation,
and we claim that 
\[\bigcap_{\alpha \in \Phi-\Psi} \ker\alpha = \{0\}.\]
Indeed, let $h = x_1h_1 + \dots + x_{n-1}h_{n-1} \in H$ with $h \neq 0$. 
By taking $i$ such that $x_i \neq -\sum_{j=1}^{n-1} x_j$, 
we have $\sigma +\varepsilon_i \in \Phi - \Psi$ and 
$(\sigma +\varepsilon_i)h = x_i+\sum_{j=1}^{n-1} x_j \neq 0$, 
that is, $h \not \in \ker (\sigma +\varepsilon_i)$, and the claim follows.

Finally, let $I$ be a nontrivial ideal of the $\bb{Z}_p$-Lie lattice $\pi^m L$. 
By Lemma \ref{lkarinaK}, $I\not\subseteq D_\infty(\varphi_0)$,
so that $I$ is not $\varphi_m$-invariant. 
This proves that $\varphi_m$ is simple.
\end{proof}

\subsection{Self-similarity of \texorpdfstring{$\spt_{2l}(R)$}{sp\_2l(R)}}
\label{sspK}

Part (\ref{tsscllieK2}) of Theorem \ref{tsscllieK} is a consequence of Proposition \ref{pssspK} proven below.
The context is the same one of Section \ref{ssllK}, but now
we fix an integer $l\geq 1$, and we let the indices $i,j$ take values $1\les i,j\les l$.
We write matrices in block form, where the blocks have size $l\times l$.
Let 
\begin{align*}
	s &= \begin{bmatrix}
    0 & \Id_l\\
    -\Id_l & 0
\end{bmatrix}, &
    m_{ij} &= \begin{bmatrix}
    e_{ij} & 0\\
    0 & -e_{ji}
\end{bmatrix},\\[5pt]
    n_{ij} &= \begin{bmatrix}
    0 & e_{ij} + e_{ji}\\
    0 & 0
\end{bmatrix} \quad \text{for} \ i \neq j,&
    n_{ii} &= \begin{bmatrix}
    0 & e_{ii}\\
    0 & 0
\end{bmatrix},\\[5pt]
    q_{ij} &= \begin{bmatrix}
    0 & 0\\
    e_{ij} + e_{ji} & 0
\end{bmatrix} \quad \text{for} \ i \neq j,&
    q_{ii} &= \begin{bmatrix}
    0 & 0\\
    e_{ii} & 0
\end{bmatrix},
\end{align*}
and note that $n_{ij} = n_{ji}$ and $q_{ij} = q_{ji}$.
The $K$-Lie algebra 
$\spt_{2l}(K) = \{x \in \gl_{2l}(K) :\, sx + x^{\mr{t}}s = 0\}$ 
consists of matrices in the form
\[\begin{bmatrix}
    m & n\\
    q & -m^{\mr{t}}
\end{bmatrix} \qquad m,n,q \in \gl_l(K),\]
such that $n^{\mr{t}} = n$ and $q^{\mr{t}} = q$. 
The $R$-Lie lattice 
$L = \spt_{2l}(R) = \{x \in \gl_{2l}(R) :\, sx + x^{\mr{t}}s = 0\}$ 
consists of matrices in the same form with $m,n,q \in \gl_l(R)$.
The matrices $m_{ij}$, and $n_{ji}, q_{ij}$ for $i\les j$ constitute
a basis of $\spt_{2l}(K)$ over $K$ and of $\spt_{2l}(R)$ over $R$.
Observe that $L$ admits a decomposition such as the one described at 
the beginning of Section \ref{seclieclass} for the root system $\Phi = C_{l}$;
moreover, recall that the dual basis of $h_1,...,h_{l}$, where we take $h_i=m_{ii}$, 
is denoted by $\vep_1,...,\vep_{l}$. 
The root system at hand is 
\begin{gather*}
    \Phi = \{\varepsilon_i - \varepsilon_j :\,  i\neq j \} 
    \cup \{\pm(\vep_i + \varepsilon_j) :\,  i \les j\}, 
\end{gather*}
while the decomposition of $L$ reads:
$$
\begin{array}{rcll}
	H &=& \bigoplus_{i}R m_{ii}, & \\[5pt]
    L_{\varepsilon_i - \varepsilon_j} &=& R m_{ij}& \text{for}\ i \neq j, \\[3pt]
    L_{\varepsilon_i + \varepsilon_j} &=& R n_{ji} &\text{for}\ i \leq j,\\[3pt]
    L_{-\varepsilon_i - \varepsilon_j} &=& R q_{ij} &\text{for}\ i \leq j.
\end{array}
$$
Let $k\geq 1$, and
define $a = \diag(a_1, \dots, a_{2l})$ with 
\[a_i = \begin{cases}
    \pi^k & \text{if } i = l,\\
    \pi^{-k} & \text{if } i = 2l,\\
    1 & \text{otherwise.}
\end{cases}\]
Consider the $K$-Lie-algebra automorphism $\varphi(x) =axa^{-1}$ of $\spt_{2l}(K)$.
Define $M = L \cap \varphi^{-1}(L)$,
and let $m\in\bb{N}$.
Then $\pi^m M$ is an $R$-subalgebra of $\pi^mL$ of finite index, and the restriction 
$$\varphi_m:\pi^m M \to \pi^mL$$ 
of $\varphi$ is a virtual endomorphism of $\pi^mL = \spt_{2l}^m(R)$.

\begin{proposition}
\label{pssspK}
In the above context, $\varphi_m$ is a simple virtual endomorphism of index $q^{2lk}$ 
of the $\Z_p$-Lie lattice $\spt_{2l}^m(R)$.
\end{proposition}

\begin{proof}
For $i < l$, we have
\begin{align*}
    \varphi(m_{li}) &= \pi^k m_{li}, &  \varphi(m_{il}) &= \pi^{-k} m_{il},\\
    \varphi(n_{li}) &= \pi^k n_{li}, & \varphi(q_{il}) &= \pi^{-k} q_{il},\\
    \varphi(n_{ll}) &= \pi^{2k} n_{ll}, & \varphi(q_{ll}) &= \pi^{-2k} q_{ll},
\end{align*}
and $\varphi$ acts as the identity on the remaining basis elements of $L$.
We define 
$\Psi = \{\alpha \in \Phi :\, \alpha(m_{ll}) \geq 0\}$,
and we observe that
$\Phi-\Psi = 
\{\vep_i-\varepsilon_l:\, i < l\} 
\cup 
\{-\vep_i - \varepsilon_l\}
$.
Thus 
$$
M = 
H 
\oplus 
\left[\bigoplus_{\alpha\in \Psi} Re_{\alpha }\right]
\oplus 
\left[
	\bigoplus_{i<l}\left( R\pi^ke_{\vep_i-\varepsilon_l}
	\oplus
	R\pi^{k}e_{-\vep_i-\varepsilon_l}\right) 
	\right]
\oplus 
R\pi^{2k}e_{-2\vep_l},
$$ 
so that the quotient $L/M$ is isomorphic to 
$\left[\bigoplus_{i<l}R/\pi^k R\right]^{\oplus 2}\oplus R/\pi^{2k}R$
as an $R$-module.
Hence, the index of $M$ in $L$ is $q^{2lk}$,
and we note that the index of $\pi^m M$ in $\pi^m L$ has the same value. 
We have 
$$
D_\infty(\varphi_m)\subseteq D_\infty(\varphi_0)=
H 
\oplus 
\left[\bigoplus_{\alpha\in \Psi} Re_{\alpha }\right],
$$
see Remark \ref{rDinfty} for the notation,
and we claim that 
\[\bigcap_{\alpha \in \Phi-\Psi} \ker\alpha = \{0\}.\]
Indeed, suppose $h = x_1 m_{11} + \dots + x_l m_{ll} \in H$. If $x_l \neq 0$, then 
$-2\varepsilon_l \in \Phi - \Psi$ and 
$h \not\in \ker (-2\varepsilon_l)$. 
On the other hand, if $x_l = 0$ and $x_i \neq 0$ 
for some $i < l$, then $h(\varepsilon_i-\varepsilon_l)\neq 0$,
and the claim follows.

Finally, let $I$ be a nontrivial ideal of the $\bb{Z}_p$-Lie lattice $\pi^m L$. 
By Lemma \ref{lkarinaK}, $I\not\subseteq D_\infty(\varphi_0)$,
so that $I$ is not $\varphi_m$-invariant. 
This proves that $\varphi_m$ is simple.
\end{proof}

\subsection{Self-similarity of \texorpdfstring{$\so_{2l}(R)$}{so\_2l(R)}}
\label{ssoevenK}

Part (\ref{tsscllieK3}) of Theorem \ref{tsscllieK} for $n$ even is a consequence 
of Proposition \ref{psssoevenK} proven below.
The context is the same one of Section \ref{ssllK}, but now
we fix an integer $l\geq 2$, and we let the indices $i,j$ take values $1\les i,j\les l$.
We write matrices in block form, where the blocks have size $l\times l$.
Let 
\begin{align*}
	s &= \begin{bmatrix}
    0 & \Id_l\\
    \Id_l & 0
\end{bmatrix},& 
    m_{ij} &= \begin{bmatrix}
    e_{ij} & 0\\
    0 & -e_{ji}
\end{bmatrix},\\[5pt]
    n_{ij} &= \begin{bmatrix}
    0 & e_{ij} - e_{ji}\\
    0 & 0
\end{bmatrix} \quad \text{for} \ i \neq j,&
    q_{ij} &= \begin{bmatrix}
    0 & 0\\
    e_{ij} - e_{ji} & 0
\end{bmatrix} \quad \text{for} \ i \neq j,\\
\end{align*}
and note that $n_{ij} = -n_{ji}$ and $q_{ij} = -q_{ji}$.
The $K$-Lie algebra 
$\so_{2l}(K) = \{x \in \gl_{2l}(K) :\, sx + x^{\mr{t}}s = 0\}$ 
consists of matrices in the form
\[\begin{bmatrix}
    m & n\\
    q & -m^{\mr{t}}
\end{bmatrix} \qquad m,n,q \in \gl_l(K),\]
such that $n^{\mr{t}} = -n$ and $q^{\mr{t}} = -q$. 
The $R$-Lie lattice 
$L = \so_{2l}(R) = \{x \in \gl_{2l}(R) :\, sx + x^{\mr{t}}s = 0\}$ 
consists of matrices in the same form with $m,n,q \in \gl_l(R)$.
The matrices $m_{ij}$, and $n_{ji}, q_{ij}$ for $i< j$ constitute
a basis of $\so_{2l}(K)$ over $K$ and of $\so_{2l}(R)$ over $R$.
Observe that $L$ admits a decomposition such as the one described at 
the beginning of Section \ref{seclieclass} for the root system $\Phi = D_{l}$;
moreover, recall that the dual basis of $h_1,...,h_{l}$, where we take $h_i=m_{ii}$,  
is denoted by $\vep_1,...,\vep_{l}$. 
The root system at hand is 
\begin{gather*}
    \Phi = \{\varepsilon_i - \varepsilon_j :\,  i\neq j \} 
    \cup \{\pm(\vep_i + \varepsilon_j) :\,  i < j\}, 
\end{gather*}
while the decomposition of $L$ reads:
$$
\begin{array}{rcll}
	H &=& \bigoplus_{i}R m_{ii},&\\[5pt]
    L_{\varepsilon_i - \varepsilon_j} &=& R m_{ij} & \text{for}\ i \neq j, \\[3pt]
    L_{\varepsilon_i + \varepsilon_j} &=& R n_{ji} &\text{for}\ i < j,\\[3pt]
    L_{-\varepsilon_i - \varepsilon_j} &=& R q_{ij} &\text{for}\ i < j.
\end{array}
$$
Let $k\geq 1$, and
take $a = \diag(a_1, \dots, a_{2l})$ with 
\[a_i = \begin{cases}
    \pi^k & \text{if } i = l,\\
    \pi^{-k} & \text{if } i = 2l,\\
    1 & \text{otherwise.}
\end{cases}\]
Consider the $K$-Lie-algebra automorphism $\varphi(x) =axa^{-1}$ of $\so_{2l}(K)$.
Define $M = L \cap \varphi^{-1}(L)$,
and let $m\in\bb{N}$.
Then $\pi^m M$ is an $R$-subalgebra of $\pi^mL$ of finite index, and the restriction 
$$\varphi_m:\pi^m M \to \pi^mL$$ 
of $\varphi$ is a virtual endomorphism of $\pi^mL = \so_{2l}^m(R)$.

\begin{proposition}
\label{psssoevenK}
In the above context, 
$\varphi_m$ is a simple virtual endomorphism of index $q^{(2l-2)k}$ 
of the $\Z_p$-Lie lattice $\so_{2l}^m(R)$.
\end{proposition}

\begin{proof}
For $i < l$, we have
\begin{align*}
    \varphi(m_{li}) &= \pi^k m_{li}, &  \varphi(m_{il}) &= \pi^{-k} m_{il},\\
    \varphi(n_{li}) &= \pi^k n_{li}, & \varphi(q_{il}) &= \pi^{-k} q_{il},
\end{align*}
and $\varphi$ acts as the identity on the remaining basis elements of $L$.
We define 
$\Psi = \{\alpha \in \Phi :\, \alpha(m_{ll}) \geq 0\}$,
and we observe that
$\Phi-\Psi = 
\{\vep_i-\varepsilon_l:\, i < l\} 
\cup 
\{-\vep_i - \varepsilon_l: i < l\}
$.
Thus 
$$
M = 
H 
\oplus 
\left[\bigoplus_{\alpha\in \Psi} Re_{\alpha }\right]
\oplus 
\left[\bigoplus_{\alpha\in \Phi-\Psi} R\pi^ke_{\alpha }\right]
$$ 
so that the quotient $L/M$ is isomorphic to 
$\left[\bigoplus_{i<l}R/\pi^k R\right]^{\oplus 2}$
as an $R$-module.
Hence, the index of $M$ in $L$ is $q^{(2l-2)k}$,
and we note that the index of $\pi^m M$ in $\pi^m L$ has the same value. 
We have 
$$
D_\infty(\varphi_m)\subseteq D_\infty(\varphi_0)=
H 
\oplus 
\left[\bigoplus_{\alpha\in \Psi} Re_{\alpha }\right],
$$
see Remark \ref{rDinfty} for the notation,
and we claim that 
\[\bigcap_{\alpha \in \Phi-\Psi} \ker\alpha = \{0\}.\]
Indeed, 
suppose $h = x_1 m_{11} + \dots + x_l m_{ll} \in H$. 
If $x_l \neq 0$, then $\alpha := \varepsilon_1-\varepsilon_l$ and 
$\beta := -\varepsilon_1 - \varepsilon_l$ are in $\Phi -\Psi$ and 
$h \not\in \ker \alpha \cap \ker \beta$. 
Otherwise, if $x_l = 0$ and $x_i \neq 0$ for some $i < l$ 
then $\varepsilon_i-\varepsilon_l \in \Phi - \Psi$ and $h \not\in \ker (\varepsilon_i-\varepsilon_l )$.
The claim follows.

Finally, let $I$ be a nontrivial ideal of the $\bb{Z}_p$-Lie lattice $\pi^m L$. 
By Lemma \ref{lkarinaK}, $I\not\subseteq D_\infty(\varphi_0)$,
so that $I$ is not $\varphi_m$-invariant. 
This proves that $\varphi_m$ is simple.
\end{proof}

\subsection{Self-similarity of \texorpdfstring{$\so_{2l+1}(R)$}{so\_2l+1(R)}}
\label{ssooddK}

Part (\ref{tsscllieK3}) of Theorem \ref{tsscllieK} for $n$ odd is a consequence 
of Proposition \ref{psssooddK} proven below.
The context is the same one of Section \ref{ssllK}, but now
we fix an integer $l\geq 1$, and we let the indices $i,j$ take values $1\les i,j\les l$.
We write matrices in block form, where the blocks on the diagonal have sizes $1\times 1$,
$l\times l$, and $l\times l$.
For matrices of size $(2l+1)\times (2l+1)$, 
it is convenient to number the indices counting from $0$ and not $1$, as it was done implicitly
in the rest of the paper.
We denote by $e_i\in \mr{M}_{l\times 1}(R)$ the column vector with $i$-th entry $1$ and
the other entries $0$.
Let 
\begin{align*}
	s &= \begin{bmatrix}
    1 & 0 & 0\\
    0 & 0 & \Id_l\\
    0 & \Id_l & 0
\end{bmatrix},& 
    b_j &= \begin{bmatrix}
    0 & e_{j}^{\mr{t}} & 0\\
    0 & 0 & 0\\
    -e_{j} & 0 & 0
\end{bmatrix} , \\[5pt]
    c_i &= \begin{bmatrix}
    0 & 0 & -e_{i}^{\mr{t}}\\
    e_{i} & 0 & 0\\
    0 & 0 & 0
\end{bmatrix},&
    m_{ij} &= \begin{bmatrix}
    0 & 0 & 0\\
    0 & e_{ij} & 0\\
    0 & 0 & -e_{ji}
\end{bmatrix}, \\[5pt]
    n_{ij} &= \begin{bmatrix}
    0 & 0 & 0\\
    0 & 0 & e_{ij}-e_{ji}\\
    0 & 0 & 0
\end{bmatrix},&
    q_{ij} &= \begin{bmatrix}
    0 & 0 & 0\\
    0 & 0 & 0\\
    0 & e_{ij} - e_{ji} & 0
\end{bmatrix},
\end{align*}
and note that $n_{ij} = -n_{ji}$ and $q_{ij} = -q_{ji}$.
The $K$-Lie algebra 
$\so_{2l+1}(K) = \{x \in \gl_{2l+1}(K) :\, sx + x^{\mr{t}}s = 0\}$ 
consists of matrices in the form
\[\begin{bmatrix}
    0 & b & -c^{\mr{t}}\\
    c & m & n\\
    -b^{\mr{t}} & q & -m^{\mr{t}}
\end{bmatrix},\]
where $b \in \mr{M}_{1 \times l}(K)$, $c \in \mr{M}_{l \times 1}(K)$,
and $m,n,q \in \gl_l(K)$ with $n^{\mr{t}} = -n$ and $q^{\mr{t}} = -q$. 
The $R$-Lie lattice 
$L = \so_{2l+1}(R) = \{x \in \gl_{2l+1}(R) :\, sx + x^{\mr{t}}s = 0\}$ 
consists of matrices in the same form with $b,c,m,n,q$ having entries in $R$.
The matrices $b_i, c_i, m_{ij}$, and $n_{ji}, q_{ij}$ for $i< j$ constitute
a basis of $\so_{2l+1}(K)$ over $K$ and of $\so_{2l+1}(R)$ over $R$.
Observe that $L$ admits a decomposition such as the one described at 
the beginning of Section \ref{seclieclass} for the root system $\Phi = B_{l}$;
moreover, recall that the dual basis of $h_1,...,h_{l}$, where we take $h_i=m_{ii}$, 
is denoted by $\vep_1,...,\vep_{l}$. 
The root system at hand is 
\begin{gather*}
    \Phi = \{\varepsilon_i - \varepsilon_j :\,  i\neq j \} 
    \cup \{\pm(\vep_i + \varepsilon_j) :\,  i < j\}
    \cup \{\pm \vep_i\}, 
\end{gather*}
while the decomposition of $L$ reads:
$$
\begin{array}{rcll}
	H &=& \bigoplus_{i}R m_{ii},&\\[5pt]
    L_{\varepsilon_i} &=& R c_i,&\\[3pt]
    L_{-\varepsilon_j} &=& R b_j,&\\[3pt]
    L_{\varepsilon_i - \varepsilon_j} &=& R m_{ij} &\text{for}\ i \neq j, \\[3pt]
    L_{\varepsilon_i + \varepsilon_j} &=& R n_{ji} &\text{for}\ i < j,\\[3pt]
    L_{-\varepsilon_i - \varepsilon_j} &=& R q_{ij} &\text{for}\ i < j.
\end{array}
$$
Let $k\geq 1$, and
take $a = \diag(a_0, \dots, a_{2l})$ with 
\[a_i = \begin{cases}
    \pi^k & \text{if } i = l,\\
    \pi^{-k} & \text{if } i = 2l,\\
    1 & \text{otherwise.}
\end{cases}\]
Consider the $K$-Lie-algebra automorphism $\varphi(x) =axa^{-1}$ of $\so_{2l+1}(K)$.
Define $M = L \cap \varphi^{-1}(L)$,
and let $m\in\bb{N}$.
Then $\pi^m M$ is an $R$-subalgebra of $\pi^mL$ of finite index, and the restriction 
$$\varphi_m:\pi^m M \to \pi^mL$$ 
of $\varphi$ is a virtual endomorphism of $\pi^mL = \so_{2l+1}^m(R)$.

\begin{proposition}
\label{psssooddK}
In the above context,
$\varphi_m$ is a simple virtual endomorphism of index $q^{(2l-1)k}$ 
of the $\Z_p$-Lie lattice $\so_{2l+1}^m(R)$.
\end{proposition}

\begin{proof}
For $i < l$, we have
\begin{align*}
    \varphi(m_{li}) &= \pi^k m_{li}, &  \varphi(m_{il}) &= \pi^{-k} m_{il},\\
    \varphi(n_{li}) &= \pi^k n_{li}, & \varphi(q_{il}) &= \pi^{-k} q_{il},\\
    \varphi(c_l) &= \pi^k c_l, & \varphi(b_l) &= \pi^{-k} b_l,
\end{align*}
and $\varphi$ acts as the identity on the remaining basis elements of $L$.
We define 
$\Psi = \{\alpha \in \Phi :\, \alpha(m_{ll}) \geq 0\}$,
and we observe that
$\Phi-\Psi = 
\{\vep_i-\varepsilon_l:\, i < l\} 
\cup 
\{-\vep_i - \varepsilon_l: i < l\}
\cup 
\{-\vep_l\}$.
Thus 
$$
M = 
H 
\oplus 
\left[\bigoplus_{\alpha\in \Psi} Re_{\alpha }\right]
\oplus 
\left[\bigoplus_{\alpha\in \Phi-\Psi} R\pi^ke_{\alpha }\right]
$$ 
so that the quotient $L/M$ is isomorphic to 
$\left[\bigoplus_{i<l}R/\pi^k R\right]^{\oplus 2}\oplus R/\pi^k R$
as an $R$-module.
Hence, the index of $M$ in $L$ is $q^{(2l-1)k}$,
and we note that the index of $\pi^m M$ in $\pi^m L$ has the same value. 
We have 
$$
D_\infty(\varphi_m)\subseteq D_\infty(\varphi_0)=
H 
\oplus 
\left[\bigoplus_{\alpha\in \Psi} Re_{\alpha }\right],
$$
see Remark \ref{rDinfty} for the notation,
and we claim that 
\[\bigcap_{\alpha \in \Phi-\Psi} \ker\alpha = \{0\}.\]
Indeed, 
suppose $h = x_1 m_{11} + \dots + x_l m_{ll} \in H$. 
If $x_l \neq 0$, then $-\varepsilon_l \in \Phi -\Psi$ and 
$h \not\in \ker (-\varepsilon_l)$. 
Now, if $x_l = 0$ and $x_i \neq 0$ for some $i < l$
then $\varepsilon_i-\varepsilon_l \in \Phi - \Psi$ and $h \not\in \ker (\varepsilon_i-\varepsilon_l)$.
The claim follows.

Finally, let $I$ be a nontrivial ideal of the $\bb{Z}_p$-Lie lattice $\pi^m L$. 
By Lemma \ref{lkarinaK}, $I\not\subseteq D_\infty(\varphi_0)$,
so that $I$ is not $\varphi_m$-invariant. 
This proves that $\varphi_m$ is simple.
\end{proof}


\section{Proof of Theorem \ref{tssclgpK} and Theorem \ref{tssclgpK2}}
\label{secpfA}

We realize the classical groups as follows:
$$
\SL_n(K)=
\{x\in\mr{M}_n(K):\, \mr{det}(x)=1\};
$$
$$
\Sp_n(K)=
\{x\in\mr{M}_n(K):\, x^\mr{t} s x = s \},
$$
where $n$ is even and $s$ is the matrix defined
in Section \ref{sspK}; and 
$$
\SO_n(K)
=
\{x\in\SL_n(K):\, x^\mr{t} sx = s\},
$$
where $s$ is the matrix defined
in Section \ref{ssoevenK} or in Section \ref{ssooddK}, depending on the parity of $n$. 
For $\Sigma$ one of the symbols $\SL$, $\Sp$, and $\SO$, we define 
$\Sigma_n(R)=\Sigma_n(K)\cap \mr{M}_n(R)$.
By using the notation introduced above Theorem \ref{tssclgpK2},
we observe that 
the dimensions of the groups involved,
as $p$-adic analytic groups, are as follows:
$$
\begin{array}{lcccl}
\mr{dim}\,\SL_n(K) 
&=& 
(n^2-1)d
&=& 
(l^2+2l)d,
\\[7pt]
\mr{dim}\,\Sp_n(K) 
&=& 
\dfrac{(n^2+n)d}{2}
&=& 
(2l^2+l)d,
\\[10pt]
\mr{dim}\,\SO_n(K) 
&=& 
\dfrac{(n^2-n)d}{2}
&=& 
\left\{
\begin{array}{ll}
(2l^2-l)d
&\mbox{if }n\mbox{ is even},
\\[5pt]
(2l^2+l)d
&\mbox{if }n\mbox{ is odd}.
\end{array}
\right. 
\end{array} 
$$
The various subgroups of these groups considered in this paper, and their associated Lie algebras,
have the same dimensions 
(for the Lie algebras, we are talking about the dimension over $\bb{Q}_p$ or $\bb{Z}_p$,
depending on the case, and not over $K$ or $R$).

\bigskip 

We proceed with the proof of Theorem \ref{tssclgpK} by starting with a lemma.
The context is the one of the theorem, and recall that $\theta$ is defined above Theorem \ref{tsscllieK}.

\begin{lem}
Let $L$ be an $R$-Lie lattice,
and let $m\ges \theta e$. 
Then $\pi^m L$ is a powerful $\bb{Z}_p$-Lie lattice.
\end{lem}

\begin{proof}
There exists $u\in R^*$ such that $\pi^e = up$;
as a consequence, $\pi^m= \pi^{m-\theta e}u^\theta p^\theta$.
We have 
$[\pi^mL,\pi^m L] \subseteq \pi^m \pi^m L= 
p^\theta(\pi^{m-\theta e}u^\theta\pi^mL)\subseteq p^\theta(\pi^mL)$.
\end{proof}

\bigskip 
\noindent 
Since $\sigma_n^m(R)=\pi^m \sigma_n(R)$, $\sigma_n^m(R)$ is a powerful
$\bb{Z}_p$-Lie lattice for $m\ges \theta e$
(and for the values of $n$ for which $\sigma_n(R)$ is defined). 
The associated uniform pro-$p$ group may be realized through the exponential
function, as we recall. We define
$$
\epsilon = \left\lfloor \frac{e}{p-1}\right\rfloor + 1,
$$
and we observe that $\theta e\ges \epsilon$.
For $m\ges \epsilon$, the exponential and logarithm functions
give rise to mutually inverse analytic functions 
$$
\mr{exp}:
\pi^m\mr{M}_n(R)
\map 
\Id_n+\pi^m\mr{M}_n(R)
\qquad\qquad  
\mr{log}:
\Id_n+\pi^m\mr{M}_n(R)
\map 
\pi^m\mr{M}_n(R).
$$
For $m\ges e\theta$, 
$\mr{exp}(\sll_n^m(R))= \SL_n^m(R)$,
$\mr{exp}(\spt_n^m(R))= \Sp_n^m(R)$,
and
$\mr{exp}(\so_n^m(R))= \SO_n^m(R)$
are uniform pro-$p$ groups.
We recall that Proposition A of \cite{NS2019} implies that if $L$
is a powerful $\bb{Z}_p$-Lie lattice, 
$G$ is its associated uniform pro-$p$ group, and $\mr{dim}\, G\les p$ 
then $G$ is self-similar of index $p^k$ if and only if
$L$ is self-similar of index $p^k$.
Observe that, under the assumptions of the theorems that we are proving in this section,
the condition $\mr{dim}\, G\les p$ implies $p\ges 3$, hence, $\theta = 1$. 
Now, theorem \ref{tssclgpK} follows from Theorem \ref{tsscllieK} and the above considerations.

\bigskip 

Next, we prove Theorem \ref{tssclgpK2}.
Since Theorem \ref{tssclgpK} 
gives a list of self-similarity indices for $\Sigma_n^{m}(R)$ for $m\ges e$,
an analogous list for $\Sigma_n(R)$
depends on the value of 
$[\Sigma_n(R):\Sigma_n^{m}(R)]$,
by Corollary \ref{cssclgpK}.
We first argue for $\SL$ and $\Sp$, 
where we make considerations that hold for $m\ges 1$.
The function 
$\Sigma_n(R)\rightarrow \Sigma_n(R/\pi^mR)$
of reduction modulo $\pi^m$
is a surjective group morphism with kernel $\Sigma_n^m(R)$,
so that 
$
[\Sigma_n(R):\Sigma_n^m(R)]=|\Sigma_n(R/\pi^mR)|
$;
see 
\cite[Theorem VII.6 and Theorem VII.20]{NewIM}.
The value of $|\Sigma_n(R/\pi^mR)|$
is computed in \cite[Theorem VII.15 and Theorem VII.27]{NewIM},
from which the parts of 
Theorem \ref{tssclgpK2} relative to $\SL$ and $\Sp$ follow.
On the other hand, the parts of Theorem \ref{tssclgpK2} relative to $\SO$
follow from Theorem \ref{tindSO} and Remark \ref{rindsodvr}.


\section{Indices of congruence subgroups of orthogonal groups}
\label{secindSO}

The proof of Theorem \ref{tindSO} will be given at the end of the section after the necessary
preliminaries.
The results given below, including the proof of the theorem, were inspired by and partially proved in
\cite{MSlondrina}. 
For the ease of the reader, we make complete proofs of our statements.

\begin{remark}
\label{rindsodvr}
The ring of integers of a $p$-adic field, with $p\ges 3$, satisfies
the assumptions of Theorem \ref{tindSO}.
\end{remark}

\medskip 

We start in a more general context than in the theorem. 
Let $R$ be an associative commutative ring with unit, $n\ges 1$ 
be an integer, and $l = \lfloor \frac{n}{2}\rfloor$. 
Observe that if $n$ is odd then $l\ges 0$ and $n=2l+1$, while 
if $n$ is even then $n=2l$ and $l\ges 1$.
We define 
\begin{eqnarray*}
\mr{O}_{n}(R) 
&=& 
\{M \in \mr{M}_n(R) :\,  M^\mr{t}SM = S\},
\\[2mm]
\mr{SO}_{n}(R) 
&=& 
\{M \in \mr{O}_{n}(R) :\, \det(M) = 1\},
\end{eqnarray*}
where
\[S = \begin{bmatrix}
0 & \Id_l\\
\Id_l & 0
\end{bmatrix}\]
if $n$ is even, and 
\[S = \begin{bmatrix}
0 & \Id_l & 0\\
\Id_l & 0 & 0\\
0 & 0 & 1
\end{bmatrix}\]
if $n$ is odd
(for notational convenience, when $n$ is odd we made here a different choice than in Section \ref{secpfA}
regarding the definition of $\SO_n$; it is clear that the two choices give rise to isomorphic groups).
Then $\mr{O}_{n}(R)$ is a group under matrix multiplication, $\SO_n(R)$ is a subgroup of $\mr{O}_n(R)$, and 
$S \in \mr{O}_n(R)$; observe that $S^2 = \mr{Id}$ and $\mr{det}(S)=(-1)^l$. 
Throughout this section, the elements of $R^n$ will be presented as $n \times 1$ matrices. 
The canonical basis of $R^n$ is denoted by $e_1, \dots, e_n$ as usual,
and we use the same notation, for instance, for the canonical basis
of $(R/I)^n$, where $I$ is an ideal of $R$.

\bigskip 

The following lemma is a straightforward computation.

\begin{lem}\label{conditions for orthogonal matrixes}
Suppose $l \geq 1$, and 
let $A$, $B$, $C$, $D$, $U$, $V$, $W$ be matrices in 
$\mr{M}_{l \times l}(R)$, $u,v \in R^l$, $x,y \in M_{1 \times l}(R)$ and $\omega \in R$. 
Then the following holds.

\begin{enumerate}
	\item
	\begin{enumerate}
		\item $
		\begin{bmatrix}
			A & B\\
			C & D
		\end{bmatrix}
		\in \mr{O}_{2l}(R) \iff 
		\left\{\begin{aligned}
			&A^\mr{t}C + C^\mr{t}A = 0,\\
			&A^\mr{t}D + C^\mr{t}B = \Id_l,\\
			&B^\mr{t}D + D^\mr{t}B = 0.
		\end{aligned}\right.$ \label{general condition n = 2l}
		\item $
		\begin{bmatrix}
			U & 0\\
			0 & V
		\end{bmatrix}
		\in \mr{O}_{2l}(R) \iff U^\mr{t} = V^{-1}$.
		\item $
		\begin{bmatrix}
			\Id_l & 0\\
			W & \Id_l
		\end{bmatrix}
		\in \mr{O}_{2l}(R) \iff W + W^\mr{t} = 0$.
	\end{enumerate}
	
	\item
	\begin{enumerate}
		\item $
		\begin{bmatrix}
			A & B & u\\
			C & D & v\\
			x & y & \omega
		\end{bmatrix}
		\in \mr{O}_{2l+1}(R) \iff 
		\left\{\begin{aligned}
			&A^\mr{t}C + C^\mr{t}A + x^\mr{t}x = 0,\\
			&A^\mr{t}D + C^\mr{t}B + x^\mr{t}y = \Id_l,\\
			&B^\mr{t}D + D^\mr{t}B + y^\mr{t}y = 0,\\
			&A^\mr{t}v + C^\mr{t}u + \omega x^\mr{t} = 0,\\
			&B^\mr{t}v + D^\mr{t}u + \omega y^\mr{t} = 0,\\
			&u^\mr{t}v + v^\mr{t}u + \omega^2 = 1.
		\end{aligned}\right.$ \label{condition 1. n = 2l+1}
		\item $
		\begin{bmatrix}
			U & 0 & 0\\
			0 & V & 0\\
			0 & 0 & \omega
		\end{bmatrix}
		\in \mr{O}_{2l+1}(R) \iff  		 
		\left\{\begin{aligned}
			&U^\mr{t} = V^{-1},\\
			&\omega^2 = 1.
		\end{aligned}\right.$
		\item $
		\begin{bmatrix}
			\Id_l & 0 & 0\\
			W & \Id_l & v\\
			x & 0 & \omega
		\end{bmatrix}
		\in \mr{O}_{2l+1}(R) \iff 
		\left\{\begin{aligned}
			&W + W^\mr{t} = - x^\mr{t}x,\\
			&v = -\omega x^\mr{t},\\
			&\omega^2 = 1.
		\end{aligned}\right.$
	\end{enumerate}
\end{enumerate}
\end{lem}

\begin{lem}
\label{llocalnot2}
Assume that $R$ is local and with residue field of characteristic different from 2.
Then the following holds.
\begin{enumerate}
	\item 
	\label{llocalnot2.1}
	Assume that $a\in R$ and $a^2 = 1$. Then $a=1$ or $a=-1$.
	\item 
	\label{llocalnot2.2}
	$[\mr{O}_{n}(R) : \mr{SO}_{n}(R)] = 2$.
	\item 
	\label{llocalnot2.3}
	$\mr{O}_1(R) = \{1,-1\}$.
	\item 
	\label{llocalnot2.4}
	$\mr{O}_2(R) = \left\{ \begin{bmatrix} \alpha & 0\\ 0 & \alpha^{-1}\end{bmatrix} :\, \alpha \in R^*\right\} \cup \left\{\begin{bmatrix} 0 & \alpha^{-1}\\ \alpha & 0 \end{bmatrix} :\, \alpha \in R^*\right\}$.
\end{enumerate}
\end{lem}

\begin{proof}
Observe that $2\in R^*$.

(\ref{llocalnot2.1})
Assume by contradiction that $a\neq \pm 1$. 
Then the condition $(a-1)(a+1)=0$ implies
$a-1,a+1\not\in R^*$, so that $2 = (a+1)-(a-1)\not\in R^*$, since $R$ is local, a contradiction.

(\ref{llocalnot2.2})
It is sufficient to show that $\det(\mr{O}_{n}(R)) = \{1,-1\}$. 
The inclusion `$\subseteq$' follows from item (\ref{llocalnot2.1}). 
To prove the inclusion `$\supseteq$', we exhibit a matrix $M \in \mr{O}_{n}(R)$ such that $\det(M) = -1$: 
\[M = \begin{bmatrix}
0 & 0 & 1 & 0\\
0 & \Id_{l-1} & 0 & 0\\
1 & 0 & 0 & 0\\
0 & 0 & 0 & \Id_{l-1}
\end{bmatrix}\]
for $n$ even and 
\[M = \begin{bmatrix}
\Id_{2l} & 0\\
0 & -1
\end{bmatrix}\]
for $n$ odd.

(\ref{llocalnot2.3})
Since $\mr{O}_1(R)=\{a\in R:\, a^2 =1\}$,
the item follows from item (\ref{llocalnot2.1}).

(\ref{llocalnot2.4}) 
A matrix 
$
\begin{bmatrix} a & b\\ c & d\end{bmatrix}
\in \mr{M}_2(R)$ is in $\mr{O}_2(R)$ if and only if $ad+bc=1$, $ac =0$, and $bd = 0$.  
The inclusion `$\supseteq$' is readily verified.
For `$\subseteq$' one may argue as follows. If one of the entries is zero, it is easy to 
see that the matrix has one of the two forms specified in the statement. 
If none of the entries is zero we get a contradiction by observing that none of them may be invertible,
so that $1$ would be in the maximal ideal of $R$. 
\end{proof}

\bigskip 

For a vector $x \in R^n$, we denote by $(x)$ the ideal of $R$ generated by the entries of $x$.
We leave the proof of the next lemma to the reader.

\begin{lem}
\label{lem2}
Let $x \in R^n$, 
and consider the following conditions.
\begin{enumerate}
\item
\label{lem2.1} 
$(x) = R$.
\item 
\label{lem2.2} 
$\exists i\in \{1,...,n\}$ such that $x_i\in R^*$.
\item 
\label{lem2.3} 
$\exists M\in \GL_n(R)$ such that $Me_1=x$.
\end{enumerate}
Then (\ref{lem2.2}) $\Rightarrow$ (\ref{lem2.3}) $\Rightarrow$ (\ref{lem2.1}).
Moreover, if $R$ is local then the three conditions are equivalent.
\end{lem}

\begin{remark}
\label{cons}
We make a construction that will be used in a couple of lemmas below.
We assume that $n\ges 2$, so that $l\ges 1$. 
Let $x\in R^n$ be such that $x^\mr{t}Sx = 0$,
and let $U\in \GL_l(R)$ be such that 
$$
Ue_1 = \begin{bmatrix}
x_1\\
\vdots\\
x_l
\end{bmatrix}.
$$ 
Such a $U$ may not exists (cf. Lemma \ref{lem2}), but in case it exists,
as it is assumed here, we will construct $M\in\mr{O}_n(R)$
such that $Me_1 = x$.
Define 
$$
y = U^\mr{t}\begin{bmatrix}
x_{l+1}\\
\vdots\\
x_{2l}
\end{bmatrix}.
$$
Observe that $y_1 = \sum_{i=1}^l x_i x_{l+i}$, so that
\[0 = x^\mr{t}Sx = 
\begin{cases} 
2y_1 & \text{if } n = 2l,\\
2y_1 + x_{n}^2 & \text{if } n = 2l+1.
\end{cases}\] 
We define a matrix in $M_l(R)$ by
$$
W = 
\begin{bmatrix}
y_1 & -y_2 & \dots & -y_l\\
y_2 & 0 & \dots & 0 \\
\dots & \dots & \dots & \dots\\
y_l & 0 &\dots & 0
\end{bmatrix},
$$
and observe that $We_1 = y$ and 
\[
W + W^\mr{t} = 
\begin{cases} 
0 & \text{if } n = 2l,\\
-x_n^2e_1e_1^\mr{t} & \text{if } n = 2l+1.
\end{cases}
\]
Define
\[M = \begin{bmatrix}
U & 0\\
0 & U^{-1,\mr{t}}
\end{bmatrix}
\begin{bmatrix}
\Id_l & 0\\
W & \Id_l
\end{bmatrix}\]
for $n$ even, and 
\[M = \begin{bmatrix}
U & 0 & 0\\
0 & U^{-1,\mr{t}} & 0\\
0 & 0 & 1
\end{bmatrix}
\begin{bmatrix}
\Id_l & 0 & 0\\
W & \Id_l & -x_ne_1\\
x_ne_1^\mr{t} & 0 & 1
\end{bmatrix}\]
for $n$ odd. 
The reader may check that $M$ satisfies the desired properties.
\end{remark}

\begin{lem}
\label{prop3}
Let $n \geq 2$ and $x \in R^{n}$,
and consider the following conditions.
\begin{enumerate}
\item
\label{prop3.1} 
$(x) = R$ and $x^\mr{t}Sx = 0$.
\item 
\label{prop3.2} 
$\exists M\in \mr{O}_n(R)$ such that $Me_1=x$.
\end{enumerate}
Then (\ref{prop3.2}) $\Rightarrow$ (\ref{prop3.1}).
Moreover, if $R$ is local then the two conditions are equivalent.
\end{lem}

\begin{proof}
We first assume (\ref{prop3.2}) and prove (\ref{prop3.1}).
The condition $(x)=R$ follows from Lemma \ref{lem2}.
Moreover, $x^\mr{t}Sx = e_1^\mr{t}M^\mr{t}SMe_1 = e_1^\mr{t}Se_1 = 0$, 
the last equality being true since $n\ges 2$.

We now assume (\ref{prop3.1}) and that $R$ is local. 
Moreover, observe that $l\ges 1$.
By definition of $S$, 
\[x^\mr{t}Sx = 
\begin{cases} 
\sum_{i=1}^l 2 x_i x_{l+i} & \text{if } n = 2l,\\[3pt]
\sum_{i=1}^l 2 x_i x_{l+i} + x_{n}^2 & \text{if } n = 2l+1.
\end{cases}\]
By Lemma \ref{lem2}, there is some $ 1\leq i_0 \leq n$ such that $x_{i_0} \in R^*$. 
We claim we can take $1 \leq i_0 \leq 2l$. In fact, for $n = 2l+1$ 
the equation $x^\mr{t}Sx = 0$ implies $(x_n^2) \subseteq (x_1,\dots, x_l)$, 
thus if $x_n \in R^*$ then $(x_1,...,x_l)= R$, and we can take $i_0 \leq l$
by the same lemma. 
In case $i_0 \leq l$ one can apply Lemma \ref{lem2} (with $l$ in place of $n$) 
and Remark \ref{cons} to get the desired $M$.
Finally, we  treat the missing case, $l +1 \leq i_0 \leq 2l$, 
by applying the previous one to $Sx$.
More precisely, since $(Sx) = (x) = R$ and $(Sx)^\mr{t}S(Sx) = x^\mr{t}Sx = 0$, 
there is $A \in \mr{O}_{n}(R)$ such that $Ae_1 = Sx$. 
Then $M = SA$ has the desired properties.
\end{proof}

\bigskip

Let $I$ be an ideal of $R$. 
The canonical projection $\rho_I : R \to R/I$ induces an $R$-module morphism 
$R^n \to (R/I)^n$, an $R$-algebra morphism $\mr{M}_n(R) \to \mr{M}_n(R/I)$, 
and group morphisms 
$\mr{O}_n(R) \to \mr{O}_n(R/I)$
and
$\SO_n(R) \to \SO_n(R/I)$, 
all of which will be denoted by $\rho_I$ or simply by $\rho$, especially in the proofs.
We define
\begin{eqnarray*}
\mr{O}_n(R,I) 
&=& 
\ker(\rho_I : \mr{O}_n(R) \to \mr{O}_n(R/I)),
\\
\mr{SO}_{n}(R,I) 
&=& 
\ker(\rho_I: \mr{SO}_{n}(R)\to \mr{SO}_{n}(R/I)).
\end{eqnarray*}
Observe that
\[[\mr{O}_n(R) : \mr{O}_n(R,I)] = | \rho_I(\mr{O}_n(R))|;\]
moreover, comparing with the notation of Theorem \ref{tindSO}, 
we have 
$\mr{O}_n^m(R)=\mr{O}_n(R, P^m)$ 
and
$\SO_n^m(R)=\SO_n(R, P^m)$. 

\begin{remark}
\label{rr1}
Assume that $R$ is local with residue field of characteristic different from 2,
and $I\neq R$. Then $[\mr{O}_1(R) : \mr{O}_1(R,I)] = 2$; see Lemma \ref{llocalnot2}.
\end{remark}

\begin{lem}\label{SO_n(R) is O_n(R)}
Assume that $R$ is local and with residue field of characteristic different from 2,
and  $I\neq R$.
Then $\mr{SO}_{n}(R,I) = \mr{O}_{n}(R,I)$.
\end{lem}

\begin{proof}
The inclusion `$\subseteq$' is trivial. 
Assume that $M \in \mr{O}_{n}(R,I)$. 
Then $\det(M) \in \{1,-1\}$, by Lemma \ref{llocalnot2}, and $\det(M) \equiv 1$ modulo $I$. 
Since $2 \not\equiv 0$ modulo $I$, we have $\det(M) = 1$.
\end{proof} 

\bigskip 

Observe that $\mr{O}_{n}(R)$ acts on $R^{n}$ by left multiplication.
Let 
\[\mr{C}_n(R) = \{x \in R^{n} :\, \exists M\in\mr{O}_n(R)\text{ s.t. }Me_1=x\}\]
be the orbit of $e_1$ under this action.
By Lemma \ref{prop3}, if $n\ges 2$ and $R$ is local then
\[\mr{C}_n(R)=\{x \in R^{n} :\, (x) = R\ \text{and} \ x^\mr{t}Sx = 0\}.\]
For $x \in \mr{C}_n(R)$, we define 
\[T_x = \{M \in \mr{O}_{n}(R) :\, Me_1 = x\}.\]
Sometimes we will use the notation $T_n(R)$ for $T_{e_1}$,
which is  the stabilizer of $e_1$, hence, a subgroup of $\mr{O}_{n}(R)$.
Note that:
\begin{enumerate}
	\item If $A \in \mr{O}_{n}(R)$ and $x \in \mr{C}_n(R)$ then $AT_x = T_{Ax}$;
	\item $\{T_x :\, x \in \mr{C}_n(R)\} = \{AT_{e_1}:\, A \in \mr{O}_{n}(R)\}$, 
	that is, the $T_x$'s are the left cosets of $T_{e_1}$.
\end{enumerate}
 We are going to count $|\rho_I(\mr{O}_{n}(R))|$ through these cosets.

\begin{remark}
\label{rte1n2}	
Assume that, for $b\in R$, $2b =0$ implies $b=0$.
Then $T_2(R) = \{\mr{Id_2}\}$.
\end{remark}

\begin{lem}\label{lcount}
Assume that $n\ges 2$, $R$ is local, and $I\neq R$.
Then the following holds.
\begin{enumerate}
\item 
\label{lcount1}
Let $x \in \mr{C}_n(R)$ be such that $x \equiv e_1$ modulo $I$. 
Then there is some $M \in T_x$ with $M \equiv \Id_n$ modulo $I$.

\item 
\label{lcount2}
For $x,y \in \mr{C}_n(R)$ the following holds.
\begin{enumerate}
	\item \label{lcount2.1}
	Assume that $x \equiv y$  modulo $I$. 
	Then, for every $A \in T_x$, there is some $B \in T_y$ with $B \equiv A$ modulo $I$.
	\item \label{lcount2.2}
	If $\rho_I(x)=\rho_I(y)$ then $\rho_I(T_x) = \rho_I(T_y)$.
	\item \label{lcount2.3}
	If $\rho_I(x)\neq\rho_I(y)$ then $\rho_I(T_x) \cap \rho_I(T_y) = \emptyset$.
\end{enumerate}
\end{enumerate}
\end{lem} 

\begin{proof}
(\ref{lcount1})
Since $x_1 \in 1+ I \subseteq R^*$, we can take 
\[U = \begin{bmatrix}
x_1 & 0 & \dots & 0\\
\vdots & & \Id_{l-1} & \\
x_l & & & 
\end{bmatrix}
\]
in Remark \ref{cons}. By following the notation of the remark, 
note that $y \equiv 0$ modulo $I$, so that $W$ satisfies 
$W \equiv 0$ modulo $I$ and $M \equiv \Id_n$ modulo $I$.

(\ref{lcount2})
Item (\ref{lcount2.3}) holds on general grounds, while 
(\ref{lcount2.2}) follows from
(\ref{lcount2.1}).
For the proof of (\ref{lcount2.1}),
we fix $A \in T_x$. We have $A^{-1}y \equiv A^{-1}x$ modulo $I$, 
and $A^{-1}x = e_1$. Thus, 
item (\ref{lcount1})
gives $M \in T_{A^{-1}y} = A^{-1}T_y$ such that $M \equiv \Id_n$ modulo $I$. 
Then $B := AM$ has the desired properties.\phantom{aaaa}
\end{proof}

\begin{proposition}
\label{pcount}
Assume that $n\ges 2$, $R$ is local, and $I\neq R$.
Then 
$$|\rho_I(\mr{O}_{n}(R))| = |\rho_I(\mr{C}_n(R))|\cdot|\rho_I(T_n(R))|.$$
\end{proposition}

\begin{proof}
Since $|\rho(T_x)| = |\rho(T_{e_1})|$ for every $x \in \mr{C}_n(R)$,
the desired equality follows from Lemma \ref{lcount}, items
(\ref{lcount2.2}) and (\ref{lcount2.3}).
\end{proof}

\begin{lem}\label{lC_n(R)modI}
Assume that $n\ges 2$, $R$ is local with residue field of characteristic different from 2, 
and $I\neq R$. Then
$\rho_I(\mr{C}_n(R)) = \mr{C}_n(R/I)$ and
$|\mr{C}_n(R)| = |\mr{C}_n(R/I)|\cdot|I|^{n-1}$.
\end{lem}

\begin{proof}
The inclusion `$\subseteq$' is clear. 
Let $x \in \mr{C}_n(R/I)$. The rest of the argument
will simultaneously prove that $x\in \rho(C_n(R))$ and the stated equality of cardinalities.
Since $R/I$ is local, and by the same argument as in the proof of Lemma \ref{prop3},
there is some $1 \leq i_0 \leq 2l$ such that 
$x_{i_0} \in (R/I)^*$. 
Moreover, $\rho^{-1}(x_{i_0})\subseteq R^*$ because $R$ is local. 
So if $y \in R^n$ then
\[y \in \mr{C}_n(R) \ \text{and} \ \rho(y) = x \iff \rho(y) = x \ \text{and} \ y^\mr{t}Sy = 0.\]
Now, define 
\[i_1 = \begin{cases}
i_0 + l & \text{if }1\leq i_0 \leq l,\\
i_0 - l & \text{if }l < i_0 \leq 2l,
\end{cases}\]
thereby $2y_{i_0}y_{i_1}$ is the unique summand of $y^\mr{t}Sy$ that depends on $y_{i_0}$. 
Then 
\[\rho(y) = x \ \text{and} \ y^\mr{t}Sy = 0 
\iff 
\rho(y_i) = x_i\ \text{for} \ i \neq i_1 \ \text{and} \ y^\mr{t}Sy = 0,\] 
because, since $2x_{i_0} \in (R/I)^*$, 
$x^\mr{t}Sx = 0$ has only one solution for $x_{i_1}$ in terms of $(x_i)	_{i \neq i_1}$.
Similarly, if $y_{i_0} \in R^*$ then $y^\mr{t}Sy = 0$ 
has an unique solution for $y_{i_1}$ in terms of $(y_i)_{i \neq i_1}$.
Therefore, 
whenever $x \in \mr{C}_n(R/I)$,
the set $\{y \in \mr{C}_n(R) :\, \rho(y) = x\}$ is in bijection 
with the Cartesian power $I^{n-1}$;
in particular, it is nonempty.
\end{proof}

\begin{remark}
\label{rprinc}
Let $\pi\in R$ and $m, j\in\bb{N}$ with $j\les m$.
Consider the chain of ideals 
$R\pi^{m}\subseteq R\pi^{m-1} \subseteq \dots \subseteq R\pi\subseteq R$.
Then
$$
|R\pi^j/R\pi^{m}| = \prod_{i=j}^{m-1}|R\pi^{i}/R\pi^{i+1}|.
$$

Assume that $\pi$ is not a zero divisor,
that is, for $a\in R$, $\pi a=0$ implies $a=0$.
Then $\pi^i$ is not a zero divisor for all $i\in\bb{N}$, and
multiplication by $\pi^i$
induces an $R$-module isomorphism $R/R\pi\rar R\pi^i/R\pi^{i+1}$.
Hence,
$$
|R\pi^j/R\pi^{m}| = |R/R\pi|^{m-j}.
$$
\end{remark}

\begin{proposition} \label{formula |C_n(R/I)|}
Under the assumptions of Theorem \ref{tindSO}
and the additional assumption $n\ges 2$
we have  
\[|\mr{C}_n(R/P^m)| = 
\begin{cases}
(q^l - 1)(q^{l-1}+1)q^{(m-1)(n-1)} & \text{if } n = 2l,\\
(q^{n-1} - 1)q^{(m-1)(n-1)} & \text{if } n = 2l+1.
\end{cases}\]
\end{proposition}

\begin{proof}
First, we treat the case $m=1$, where $R/P^m= R/P \simeq \F_q$. 
For $x \in \F_q^n$, the condition $(x) = \bb{F}_q$ 
is equivalent to $x \neq 0$, thus 
\[|\mr{C}_n(\F_q)| = |\{x \in \F_q^n :\, x^\mr{t}Sx = 0\}| - 1.\]
Let $c_n = |\{x \in \F_q^n :\, x^\mr{t}Sx = 0\}|$. 
We calculate $c_n$ by recursion. 
For $n \geq 3$ and $x \in \F_q^n$, 
let $\tilde{x} \in \F_q^{n-2}$ be $x$ with the $l$-th and $2l$-th entries omitted. 
So $x^\mr{t}Sx = \tilde{x}^\mr{t}S\tilde{x} + 2x_lx_{2l}$. Since $2 \in \F_q^*$, 
if we fix $\tilde{x}$ such that $\tilde{x}^\mr{t}S\tilde{x} = 0$, 
there are $2q-1$ possible choices for $(x_l,x_{2l})$ such that $x^\mr{t}Sx = 0$. 
On the other hand, if $\tilde{x}^\mr{t}S\tilde{x} \neq 0$, then there are $q-1$ possible values for $(x_l,x_{2l})$. 
Therefore, for $n \geq 3$ we have
\begin{align*}
	c_n &= c_{n-2}(2q - 1) + (q^{n-2} - c_{n-2})(q-1)\\
	&= c_{n-2}q + q^{n-2}(q-1).
\end{align*}
Also, $c_2 = 2q-1$ and $c_1 = 1$, so we obtain
\[c_n = 
\begin{cases}
q^{l-1}(q^l + q - 1) & \text{if } n=2l,\\
q^{n-1} & \text{if } n=2l+1.
\end{cases}\]
Summarizing, we have
\[|\mr{C}_n(R/P)| = c_n - 1 =
\begin{cases}
(q^l-1)(q^{l-1}+1) & \text{if } n=2l,\\
q^{n-1} - 1 & \text{if } n=2l+1.
\end{cases}\]

Now, we treat the general case $m\ges 1$.
Note that $R/P^m$ is a local ring with residue field isomorphic to $R/P$,
so we can apply Lemma \ref{lC_n(R)modI} to $R/P^m$
and its maximal ideal $P/P^m$, obtaining
\[|\mr{C}_n(R/P^m)| = |\mr{C}_n(R/P)|\cdot|P/P^m|^{n-1}.\]
By Remark \ref{rprinc}, $|P/P^{m}|=q^{m-1}$, thus, the value of
$|\mr{C}_n(R/P^m)|$ is as in the statement. 
\end{proof}

\begin{lem}\label{l|T_e_1|}
Assume that $n \geq 2$ and $2\in R^*$. 
Then 
$|\rho_I(T_n(R))| = |\rho_I(\mr{O}_{n-2}(R))|\cdot |R/I|^{n-2}$.
(We define $\mr{O}_0(R)$ to be a trivial group, for all $R$). 
\end{lem}

\begin{proof}
Let $M \in \mr{M}_n(R)$ be such that $Me_1 = e_1$. 
Notationally, we treat the case where $n$ is odd,
and we write
\[
M = 
\begin{bmatrix}
1 & a & \lambda & b & \alpha\\
0 & A & x & B & u\\
0 & c & \mu & d & \beta\\
0 & C & y & D & v\\
0 & e & \nu & f & \gamma
\end{bmatrix},
\]
where the matrices on the diagonal have sides 1, $l-1$, 1, $l-1$ and 1, respectively. 
For even $n$, one has to neglect the last row and column, and their parameters.
For $n=2,3$, one (also) neglects the second and fourth rows and columns together with their parameters.
By Lemma \ref{conditions for orthogonal matrixes},
$M \in T_{e_1}$ if and only if
\begin{enumerate}
	\item $\mu = 1$, $\beta = 0$ and $c = d = 0$,
	\item $a = -y^\mr{t}A - x^\mr{t}C - \nu e$,
	\item $2\lambda = -y^\mr{t}x - x^\mr{t}y - \nu^2$,
	\item $b = -y^\mr{t}B - x^\mr{t}D - \nu f$,
	\item $\alpha = -y^\mr{t}u - x^\mr{t}v -\gamma\nu$, and
	\item $\begin{bmatrix}
	A & B & u\\
	C & D & v\\
	e & f & \gamma
	\end{bmatrix} \in \mr{O}_{n-2}(R)$.
\end{enumerate}
Hence,  we have exactly the freedom to choose $x,y \in R^{l-1}$, 
$\begin{bmatrix}
	A & B & u\\
	C & D & v\\
	e & f & \gamma
	\end{bmatrix} \in \mr{O}_{n-2}(R)$,
and $\nu\in R$ (as indicated above, the last condition has to be neglected for
even $n$).   
Observe that $2\in (R/I)^*$, so that the above arguments apply to $R/I$ as well. 
Consider the diagram
$$
\xymatrix{
\mr{O}_{n-2}(R)\times R^{n-2} 
\ar[r]\ar[d]
&
\mr{O}_{n-2}(R/I)\times (R/I)^{n-2} 
\ar[d]
\\
T_n(R) 
\ar[r]
&
T_n(R/I),
}
$$
where the horizontal maps are induced by reduction modulo $I$,
and the vertical maps are the bijections induced by the arguments above.
Since the diagram commutes and the image of the upper horizontal map is 
$\rho(\mr{O}_{n-2}(R))\times (R/I)^{n-2}$,
the  equality of cardinalities claimed in the statement follows.
\end{proof}

\bigskip 

\noindent
\textbf{Proof of Theorem \ref{tindSO}.}
For the nonspecial orthogonal groups the proof relies on 
Proposition \ref{pcount}, 
Lemma \ref{lC_n(R)modI}, 
Proposition \ref{formula |C_n(R/I)|}, 
Lemma \ref{l|T_e_1|},
Remark \ref{rprinc},
Remark \ref{rr1},
and Remark \ref{rte1n2};
see the details below.
The results for the special orthogonal groups
follow from the ones for the orthogonal groups
together with
Lemma \ref{llocalnot2}(\ref{llocalnot2.2})
and 
Lemma \ref{SO_n(R) is O_n(R)}.
In what follows we denote $[\mr{O}_{n}(R) : \mr{O}_{n}(R,P^m)]= |\rho(\mr{O}_n(R))|$
by $r_{n,m}$. 
 
We first treat the case where $n = 2l$.
If $n \geq 4$ we have
\begin{align*}
r_{n, m} &= (q^l - 1)(q^{l-1}+1)q^{(m-1)(n-1)}q^{m(n-2)}r_{n-2, m}\\
&= \left(1 - \frac{1}{q^l}\right)\left(1+\frac{1}{q^{l-1}}\right)q^{m(2n-3)}r_{n-2, m}\\
&= \left(1 - \frac{1}{q^l}\right)\left(1+\frac{1}{q^{l-1}}\right)q^{m(4l-3)}r_{n-2, m}.
\end{align*}
Since 
\[r_{2, m} = |\mr{C}_2(R/P^m)| = 2(q-1)q^{m-1} = 2q^m\left(1 - \frac{1}{q}\right),\]
we have
\begin{align*}
r_{n, m} &= 2q^m\left(1 - \frac{1}{q}\right)\prod_{i=2}^{l}\left[\left(1 - \frac{1}{q^i}\right)\left(1+\frac{1}{q^{i-1}}\right)q^{m(4i-3)}\right]\\
&= 2q^{m(2l^2-l)}\left(1 - \frac{1}{q^l}\right)\prod_{i=1}^{l-1}\left[\left(1 - \frac{1}{q^i}\right)\left(1+\frac{1}{q^{i}}\right)\right]\\
&= 2q^{m(2l^2-l)}\left(1 - \frac{1}{q^l}\right)\prod_{i=1}^{l-1}\left(1 - \frac{1}{q^{2i}}\right).
\end{align*}

For the case $n = 2l+1$,
we take $n \geq 3$ and compute
\begin{align*}
r_{n,m} &= (q^{n-1} - 1)q^{(m-1)(n-1)}q^{m(n-2)}r_{n-2,m}\\
&= \left(1 - \frac{1}{q^{n-1}}\right)q^{m(2n-3)}r_{n-2,m}\\
&= \left(1 - \frac{1}{q^{2l}}\right)q^{m(4l-1)}r_{n-2,m}.
\end{align*}
Since $r_{1,m} = 2$,
\begin{flalign*}
&&r_{n,m} &= 2\prod_{i=1}^{l}\left[\left(1 - \frac{1}{q^{2i}}\right)q^{m(4i-1)}\right]\\
&&&= 2q^{m(2l^2+l)}\prod_{i=1}^{l}\left(1 - \frac{1}{q^{2i}}\right).&&\square 
\end{flalign*}


\begin{footnotesize}

\providecommand{\bysame}{\leavevmode\hbox to3em{\hrulefill}\thinspace}
\providecommand{\MR}{\relax\ifhmode\unskip\space\fi MR }
\providecommand{\MRhref}[2]{%
  \href{http://www.ams.org/mathscinet-getitem?mr=#1}{#2}
}
\providecommand{\href}[2]{#2}

\end{footnotesize}

\vspace{5mm}

\noindent
Karina Livramento\\
Mathematics Institute\\
Federal University of Rio de Janeiro\\
Av. Athos da Silveira Ramos, 149\\
21941-909, Rio de Janeiro, RJ, Brazil  \\
{\tt karina@im.ufrj.br}\\

\noindent
Francesco Noseda\\
Mathematics Institute\\
Federal University of Rio de Janeiro\\
Av. Athos da Silveira Ramos, 149\\
21941-909, Rio de Janeiro, RJ, Brazil  \\
{\tt noseda@im.ufrj.br}\\

\end{document}